\documentclass[11pt]{amsart}

\usepackage{amssymb,latexsym, mathtools,extpfeil,amsmath,comment}
\usepackage{amsfonts}
\usepackage{multirow}
\usepackage{dsfont}
\usepackage{youngtab,ytableau}
\usepackage{color, booktabs}
\usepackage{graphicx}
\usepackage[all]{xy}
\usepackage{tikz-cd}
\usetikzlibrary{cd}
\usepackage{tikz}

\usepackage{scalerel,stackengine}
\stackMath
\newcommand\reallywidehat[1]{%
\savestack{\tmpbox}{\stretchto{%
  \scaleto{%
    \scalerel*[\widthof{\ensuremath{#1}}]{\kern.1pt\mathchar"0362\kern.1pt}%
    {\rule{0ex}{\textheight}}
  }{\textheight}%
}{2.4ex}}%
\stackon[-6.9pt]{#1}{\tmpbox}%
}

\numberwithin{equation}{section}

\newtheorem{thm}{Theorem}[section]
\newtheorem{lem}[thm]{Lemma}

\newtheorem{prop}[thm]{Proposition}

\theoremstyle{remark}
\newtheorem{rem}[thm]{Remark}
\newtheorem{example}[thm]{Example}
\theoremstyle{definition}

\theoremstyle{definition}

\theoremstyle{definition}

\numberwithin{equation}{section}

\newcommand{\ds}{\displaystyle}

\newcommand{\Z}{\mathbb Z}

\newcommand{\C}{\mathbb C}

\DeclareMathOperator{\Gr}{Gr}
\DeclareMathOperator{\Span}{span}

\def\presuper#1#2%
  {\mathop{}%
   \mathopen{\vphantom{#2}}^{#1}%
   \kern-\scriptspace%
   #2}

\ytableausetup{boxsize=1.1em}

\begin{document}

\title[Partition Intervals]{On the generating function for intervals in Young's lattice}
\author{Faqruddin Ali Azam}
\address{Department of Mathematics\\ Oklahoma State University  \\ Stillwater, Oklahoma 74078 \\ U.S.A. }
\email{sazam@okstate.edu}

\author{Edward Richmond}
\address{Department of Mathematics\\ Oklahoma State University  \\ Stillwater, Oklahoma 74078 \\ U.S.A. }
\email{edward.richmond@okstate.edu}


\begin{abstract}
In this paper, we study a family of generating functions whose coefficients are polynomials that enumerate partitions in lower order ideals of Young's lattice.  Our main result is that this family satisfies a rational recursion and are therefore rational functions.  As an application, we calculate the asymptotic behavior of the cardinality of lower order ideals for the ``average" partition of fixed length and give a homological interpretation of this result in relation to Grassmannians and their Schubert varieties.
\end{abstract}

\maketitle
 %

\section{introduction}
\ytableausetup {boxsize=0.75em}
\ytableausetup
{aligntableaux=top}


  Young's lattice is the set of integer partitions partially ordered by containment of Young diagrams.  Let $\Lambda$ denote the set of integer partitions and for any  $\lambda\in \Lambda$, consider the lower order ideal $[\emptyset,\lambda]$ in Young's lattice.   Lower order ideals have been studied extensively and are important objects in algebraic combinatorics, representation theory and algebraic geometry.  For example, maximal chains in $[\emptyset,\lambda]$ correspond to standard Young tableaux of shape $\lambda$.  In representation theory of the symmetric group, standard Young tableaux index a basis for irreducible representations.  In algebraic geometry, these tableaux are used to characterize the irreducible components of Springer fibers in the flag variety \cite{Fresse-Melnikov15}.  For more on Young's lattice, Young tableaux and their applications see \cite{Fulton97,Macdonald15}.

For a partition $\lambda=(\lambda_1\geq \lambda_2\geq\cdots \geq \lambda_k>0)$, let $|\lambda|:=\sum_i\lambda_i$ denote its rank and $k$ denote its length.  The rank generating function of the interval $[\emptyset,\lambda]$ is defined as
$$
P_{\lambda}(y):=\sum_{\mu\in[\emptyset,\lambda]} y^{|\mu|}.
$$
For example, if $\lambda=(2,1)$, then $P_{(2,1)}(y)=1+y+2y^2+y^3$ as seen from the Hasse diagram of $[\emptyset,(2,1)]$:

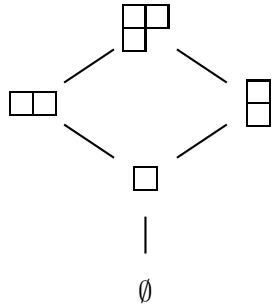
\begin{center}
\begin{tikzpicture}[scale=0.5]
\draw[thick] (0,-6)--(0,-9);
\draw[thick] (0,-2)--(-3,-4);
\draw[thick] (0,-2)--(3,-4);
\draw[thick] (0,-6)--(-3,-4);
\draw[thick] (0,-6)--(3,-4);
\draw[white, fill=white] (0,-2) circle (6ex);
\draw[white, fill=white] (-3,-4) circle (6ex);
\draw[white, fill=white] (3,-4) circle (6ex);
\draw[white, fill=white] (0,-6) circle (6ex);
\draw[white, fill=white] (0,-9) circle (6ex);
\draw (0,-2) node {\begin{ytableau}
{}&  \\
{}
\end{ytableau}};
\draw (-3,-4) node {\begin{ytableau}
{}&
\end{ytableau}};
 \draw (3,-4) node {\begin{ytableau}
{}    \\
{}
\end{ytableau}};
\draw (0,-6) node {\begin{ytableau}
{}
\end{ytableau}};
 \draw (0,-9) node{$\emptyset$};
\end{tikzpicture}
\end{center}
In the case that $\lambda=(n^k)$ (i.e. $\lambda$ has a Young diagram with rectangular shape), the corresponding rank generating function is the Gaussian polynomial
\begin{equation}\label{Eq:Gaussian}
P_{(n^k)}(y)=\binom{n+k}{k}_y:=\frac{(1-y^{n+1})\cdots(1-y^{n+k})}{(1-y)\cdots(1-y^{k})}.
\end{equation}
Gaussian polynomials are palindromic, unimodal polynomials which generalize binomial coefficients \cite{Andrews99}.  Taking the limit as $y\rightarrow 1$ yields
$$P_{(n^k)}(1)=\#[\emptyset,(n^k)]=\binom{n+k}{k}$$
and hence, asymptotically, we have $\ds\#[\emptyset,(n^k)]\sim \frac{n^k}{k!}$ as $n\rightarrow \infty.$  For arbitrary partitions, Hanada-Monhanty \cite{Handa-Mohanty80} and Ueno \cite{Ueno88} showed that $P_{\lambda}(y)$ can be calculated as a determinant of Gaussian polynomials generalizing Equation \eqref{Eq:Gaussian} (See also \cite{Gessel-Loehr10}).  Despite the existence of such a nice formula, the behavior of $P_{\lambda}(y)$ is difficult to predict. If the Young diagram of $\lambda$ is not rectangular, then $P_{\lambda}(y)$ is not palindromic.  In \cite{Stanton90}, Stanton shows that $P_{\lambda}(y)$ is not necessarily unimodal.

The goal of this paper is to study the polynomials $P_{\lambda}(y)$ for partitions of a fixed length by collecting them into a single generating function.  We prove that this generating function is rational and, as an application, we determine the asymptotic behavior of the cardinality $\#[\emptyset,\lambda]$ for the ``average partition" of length $k$ as $|\lambda|\rightarrow\infty$.  We also describe these results in the context of singular homology of the Grassmannian and its Schubert varieties.


\section{Main results}\label{S:main_results}

Let $\Lambda(k)$ denote the set of partitions of length $k$ and define the generating series
$$Q_k(x_1,\ldots,x_k,y):=\sum_{\lambda\in\Lambda(k)} P_\lambda(y)\cdot \textbf{x}^\lambda$$
where $\textbf{x}^\lambda:=x_1^{\lambda_1}\cdots x_k^{\lambda_k}.$  Our first result is that the function $Q_k$ satisfies a rational recursion and hence is a rational function.  Define the square-free monomial $$p_k:=x_1\cdots x_k$$ and for a sequence of parameters $Z=(z_1,\ldots,z_{k+1}),$ let $$Q_k(Z):=Q_k(z_1,\ldots,z_{k+1}).$$  If $Z=(x_1,\ldots,x_k,y)$ is the standard set of variables, then we denote $Q_k=Q_k(Z).$
\begin{thm}\label{T:main1}
The function $Q_0=1$ and for $k\geq 1$, we have
\begin{equation}\label{Eq:main_recursion}
\left(1-p_k \right)Q_k = x_kQ_{k-1}+\sum_{0\leq i<r\leq k}\left(\frac{y^r p_k}{1-y^r p_r}\right) Q_{k-r}(Z_r)\cdot Q_i,
\end{equation}
 where $Z_r:= \left(y^rp_{r+1},x_{r+2},x_{r+3},\cdots,x_k,y \right).$  In particular, $Q_k$ is a rational function in the variables $x_1,\ldots,x_k, y.$
\end{thm}
\begin{example}
Using Theorem \ref{T:main1}, we see that
$$Q_1(x_1,y)=\frac{x_1+(x_1-x_1^2)y}{(1-x_1)(1-x_1y)}$$ and
$$Q_{2}(x_1, x_2, y) =  \frac{x_1x_2 + (x_1x_2- x_1^2x_2-x_1^2x_2^2)y + (x_1x_2-x_1^2x_2^2 )y^2 + (x_1^3x_2^3 - x_1^2x_2^2)y^3 } {(1-x_1)(1-x_1x_2)(1-x_1y)(1-x_1x_2y)(1-x_1x_2y^2)}.$$
\end{example}

Observe that if we set $y=0$, then Equation \eqref{Eq:main_recursion} implies
\begin{equation}\label{Eq:y_is_0_recursion}
Q_k(x_1,\ldots,x_k,0)=\sum_{\lambda\in\Lambda(k)} \textbf{x}^\lambda=\frac{x_k}{(1-p_k)}\cdot Q_{k-1}(x_1,\ldots,x_{k-1},0).
\end{equation}
Thus we get the well-known formula
$$\sum_{\lambda\in\Lambda(k)} \textbf{x}^\lambda=p_k\cdot\prod_{m=1}^k(1-p_m)^{-1}.$$
Let $c_{k,n}$ denote the number of partitions in $\Lambda(k)$ of rank $n.$  If we further set $x=x_1=x_2=\cdots=x_k$ and $y=0$, then
$$Q_k(x,\ldots,x,0)=\sum_{\lambda\in\Lambda(k)} x^{|\lambda|}=\sum_{n=k}^{\infty} c_{k,n}\, x^n.$$
Applying Equation \eqref{Eq:y_is_0_recursion} recovers the classical formula that
$$\sum_{n=k}^{\infty} c_{k,n}\, x^n=x^k\cdot\prod_{m=1}^k(1-x^m)^{-1}.$$
We now consider the deformation of $Q_k$ where we set $x=x_1=x_2=\cdots=x_k$ and $y=1$.  Define the coefficients $C_{k,n}$ by
$$Q_k(x,\ldots,x,1)=\sum_{n\geq k} C_{k,n}\, x^n.$$
By definition, the coefficients $$\ds C_{k,n}=\sum_{\substack{\lambda\in \Lambda(k)\\ |\lambda|=n}} \#[\emptyset, \lambda].$$
Let $$A_{k,n}:=\frac{C_{k,n}}{c_{k,n}}$$ denote the average cardinality of lower order ideals for partitions of rank $n$ in $\Lambda(k)$.  Our second main result is on the asymptotic behavior of $A_{k,n}$ as $n\rightarrow\infty.$  For $k\geq 1$, define the interval
$$I_k:=[(k,k-1,\ldots,2,1), (k)^k]$$
and the constant
$$G_k:=\frac{(k-1)!}{(2k-1)!}\cdot\sum_{\lambda\in I_k}\left(\prod_{i=1}^k\,\lambda_i^{-1}\right).$$


\begin{thm}\label{T:main2}
Let $k\geq 1.$  Then $\ds A_{k,n}\sim G_k\cdot n^k$ as $n\rightarrow \infty$.
\end{thm}

\begin{example}
For $k=3$, we have
$$G_3=\frac{1}{5\cdot 4\cdot 3}\left(\frac{1}{3\cdot 2\cdot 1}+\frac{1}{3\cdot 2\cdot 2}+\frac{1}{3\cdot 3\cdot 1}+\frac{1}{3\cdot 3\cdot 2}+\frac{1}{3\cdot 3\cdot 3}\right)=\frac{49}{6480}.$$
\end{example}
Table \ref{TBL:Gk} gives the values of $G_k$ for $k\leq 7.$
\begin{table}[h!]
    \begin{tabular}{ccccccc}
        \toprule
         $k=1$ & $k=2$ & $k=3$ & $k=4$ & $k=5$ & $k=6$ & $k=7$ \\
        \midrule
       1 &  $\frac{1}{8}$     & $\frac{49}{6480}$     & $\frac{1597}{5806080}$    &   $\frac{104797}{15552000000}$ & $\frac{30867157}{258660864000000}$ & $\frac{8883026474947}{5538476941949952000000}$\\
        \midrule
    \end{tabular}
    \caption{The first order asymptotic growth constant $G_k$.}
    \label{TBL:Gk}
\end{table}

\begin{rem}
The partitions in $I_k$ are in natural bijection with Dyck paths from $(0,0)$ to $(k,k)$.  Hence the first order growth constant $G_k$ given in Theorem \ref{T:main2} can be thought of as a weighted sum over Catalan objects.
\end{rem}

\begin{rem}\label{Rem:At_most_k}
Let $\ds \Lambda(\leq k):=\bigsqcup_{k'\leq k}\Lambda(k')$ denote the set of partitions of at most length $k$ and let $$A_{\leq k, n}:=\frac{C_{1,n}+C_{2,n}+\cdots +C_{k,n}}{c_{1,n}+c_{2,n}+\cdots +c_{k,n}}$$
denote the average cardinality of lower order ideals for partitions of rank $n$ in $\Lambda(\leq k)$.  It is easy to see that the asymptotic behavior of $\Lambda(\leq k)$ is dominated by $\Lambda(k)$ and hence
$\ds A_{\leq k,n}\sim G_k\cdot n^k$ as in Theorem \ref{T:main2}.
\end{rem}

Finally, let $$Q:=\sum_{k=0}^{\infty}Q_k=\sum_{\lambda\in\Lambda} P_\lambda(y)\cdot \textbf{x}^\lambda.$$
Note that $Q$ is a generating function in the variable $y$ and the infinite set of variables $x_1,x_2,\ldots$   Setting $y=0$, we recover the formula that
$$\sum_{\lambda\in\Lambda} \textbf{x}^\lambda=\sum_{k=0}^\infty\left( p_k\cdot\prod_{m=1}^k (1-p_m)^{-1}\right)=\prod_{k=1}^\infty (1-p_k)^{-1}$$
and further setting $x_k=x$ for all $k$ gives
$$\sum_{\lambda\in\Lambda} x^{|\lambda|}=\sum_{n=0}^\infty c_n\, x^n=\prod_{k=1}^\infty (1-x^k)^{-1}$$
where $c_n$ denotes the total number of partitions of rank $n$.  The asymptotic growth of $c_n$ was famously computed by Hardy and Ramanujan in \cite{Hardy-Ramanujan18}.  We do not know if Theorem \ref{T:main1} can be used to compute the asymptotic growth of the average cardinality of lower order ideals in all of Young lattice (with no restrictions on length).  We leave this as an open question.

\subsection{Connections with the geometry of Schubert varieties}\label{SS:Grassmannian}
In this section we give an overview of the connection between lower order ideals in Young's lattice with the geometry of Grassmannians and its Schubert subvarieties.  For more details on the combinatorics of Grassmannians see \cite{Fulton97}.  Let $\Gr(k,\C^n)$ denote the Grassmannian of $k$ dimensional subspaces in $\C^n$.   Fix a basis $\{e_1,\ldots, e_n\}$ of $\C^n$ and for $1\leq j\leq n$, define the subspace $E_j:=\Span \{e_1,\ldots,e_j\}$.  For any partition $\lambda$ whose Young diagram is contained in the $k\times (n-k)$ rectangle, define the Schubert cell
$$X^{\circ}_\lambda:=\{V\in\Gr(k,n)  \ |\ \dim(V\cap E_{i+\lambda_{k+1-i}})= i,\ \text{for}\ 1\leq i\leq k\}.$$
The union of Schubert cells over all partitions contained in a $k\times (n-k)$ rectangle gives a CW structure of $\Gr(k,\C^n)$.  Since the cell $X^{\circ}_\lambda$ is of real dimension $2|\lambda|$, we get that the Poincar\'{e} polynomial of $\Gr(k,\C^n)$ is Gaussian polynomial
$$\sum_{i=0}^{2k(n-k)} \dim(H_i(\Gr(k,\C^n),\Z))\ y^i=P_{(k)^{(n-k)}}(y^2)=\binom{n}{k}_{y^2}.$$
Here $H_i(\Gr(k,\C^n),\Z)$ denotes the $i$-th singular homology group of $\Gr(k,\C^n).$
The Schubert variety $X_\lambda$ is defined as the algebraic closure of $X^{\circ}_\lambda$ and the closure relations for Schubert cells is given by containment of Young diagrams.  In other words, $X^{\circ}_\mu\subseteq X_\lambda$ if and only if $\mu\in[\emptyset, \lambda]$.  Since Schubert cells also give a CW structure on Schubert varieties, homologically we have
$$\sum_{i=0}^{2|\lambda|} \dim(H_i(X_\lambda,\Z))\ y^i=P_\lambda(y^2).$$
Evaluating at $y=1$ gives
$$\#[\emptyset,\lambda]=P_\lambda(1)=\sum_{i=0}^{2|\lambda|} \dim(H_i(X_\lambda,\Z))=\dim(H_*(X_\lambda,\Z))$$
and hence $\#[\emptyset,\lambda]$ is equal to the rank of $H_*(X_\lambda,\Z)$ as a free $\Z$-module.  The following are geometric and homological interpretations of the numbers $A_{\leq k,n}$ given in Remark \ref{Rem:At_most_k} and asymptotically described in Theorem \ref{T:main2}.  Let $N$ be an integer such that $N \geq n+k.$  Then:

\begin{enumerate}
\item $A_{\leq k,n}$ is the average number of Schubert cells in the closure of a Schubert variety of complex dimension $n$ in the Grassmannian $\Gr(k,\C^N)$.

\item $A_{\leq k,n}$ is the average rank of total homology of a Schubert variety of complex dimension $n$ in the Grassmannian $\Gr(k,\C^N)$.
\end{enumerate}

\section{Recursive properties of partitions and the proof of Theorem \ref{T:main1}}
\ytableausetup {boxsize=0.75em}
\ytableausetup
{aligntableaux=top}

The goal of this section is to prove Theorem \ref{T:main1}.  We begin by giving several recursive properties of rank generating functions.  To visualize these properties we will represent a partition by its Young diagram in English notation.  When working with a pair of partitions $(\mu,\lambda)$ such that $\mu\leq\lambda$, we represent them by a skew Young diagram by shading the boxes of $\mu$ within the Young diagram of $\lambda$.

\begin{example}
For example, if $\mu=(3,2)$ and $\lambda=(5,3,2,1)$, then $(\mu,\lambda)$ is represented by the skew diagram
$$
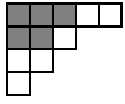
\begin{tikzpicture}
\node (a3) at (3.5,0){\begin{ytableau}
*(gray)& *(gray)& *(gray) & & \\
*(gray)& *(gray) &\\
&\\
{}
\end{ytableau}};
\end{tikzpicture}
$$
\end{example}
For any $\mu\leq\lambda$, define the polynomial
$$P_{\mu,\lambda}(y):=\sum_{\nu\in[\mu,\lambda]} y^{|\nu|}.$$
If $\mu=\emptyset$, then $P_{\emptyset,\lambda}(y)=P_{\lambda}(y)$.  Note that $P_{\mu,\lambda}(y)$ is the rank generating function of the interval $[\mu,\lambda]$ scaled by $y^{|\mu|}$.  Through out this paper, we will use the skew Young diagram of $(\mu,\lambda)$ to also represent the corresponding interval $[\mu,\lambda]$ and polynomial $P_{\mu,\lambda}(y)$.

\begin{lem}\label{L:sum} Let $\lambda\in\Lambda(k)$ and $\mu\leq\lambda.$  Let $r>0$ be an index such that $\lambda_r>\lambda_{r+1}$ (here we set $\lambda_{k+1}=0$) and define the partitions $\overleftarrow{\lambda}$ and $\overrightarrow{\mu}$ by

$$
  \overleftarrow{\lambda}_i :=
  \begin{cases}
                                   \lambda_i & \text{ if }  i\neq r \\
                                   \lambda_i-1 & \text{ if }  i=r
  \end{cases}  \qquad   \text{and}\qquad
  \overrightarrow{\mu}_i :=
  \begin{cases}
                                   \max(\mu_i,\lambda_r)& \text{if } 1\leq i\leq r \\
                                   \mu_i & \text{if } i>r
    \end{cases}.$$
Then
$$P_{\mu,\lambda}(y)=P_{\mu,\overleftarrow{\lambda}}(y)+P_{\overrightarrow{\mu},\lambda}(y).$$
\end{lem}

\begin{proof}
We partition the interval $[\mu,\lambda]$ by whether or not $\nu\in [\mu,\lambda]$ satisfies $\nu_r=\lambda_r$.  Indeed, if $\nu_r=\lambda_r$, then $\nu\geq \overrightarrow{\mu}$ which is equivalent to $\nu\in [\overrightarrow{\mu},\lambda]$.  If $\nu_r\neq\lambda_r$, then $\nu_r\leq \lambda_r-1$ which is equivalent to $\nu\in[\mu,\overleftarrow{\lambda}]$.  Thus  $[\mu,\lambda]=[\mu,\overleftarrow{\lambda}]\sqcup[\overrightarrow{\mu},\lambda]$.
\end{proof}
\begin{example}\label{example-sum}
Let $\mu=(3,2,2,1)$, $\lambda = (5,5,4,2,2)$ and $r=3$.  Lemma \ref{L:sum} implies:
$$P_{(3,2,2,1),(5,5,4,2,2)}(y)=P_{(3,2,2,1),(5,5,3,2,2)}(y)+P_{(3,3,3,1),(5,5,4,2,2)}(y).$$
Diagrammatically we have,
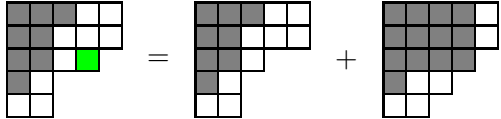
\[
\begin{tikzpicture}
\node (a2) at (2,0) {\begin{ytableau}
*(gray)&*(gray) &*(gray) &   & \\
*(gray)&*(gray) & &   & \\
*(gray) &   *(gray) &  &*(green)\\
*(gray)  &\\
&
\end{ytableau}};
\node (a2) at (3.25,0) {=};
\node (a2) at (4.5,0) {\begin{ytableau}
*(gray)&*(gray) &*(gray) &   & \\
*(gray)&*(gray) &  &   & \\
*(gray) &   *(gray) &  \\
*(gray)  &\\
&
\end{ytableau} };
\node (a2) at (5.75,0) {+};
\node (a2) at (7,0) {\begin{ytableau}
*(gray)&*(gray) &*(gray) &*(gray)  & \\
*(gray)&*(gray) &*(gray) &*(gray)   & \\
*(gray)&*(gray) &*(gray) &*(gray)    \\
*(gray) & &\\
&
\end{ytableau}};
\end{tikzpicture}
\]
Here the green box highlights the difference between $\lambda$ and $\overleftarrow{\lambda}$.
\end{example}

We remark that additive relations similar to Lemma \ref{L:sum} appear in \cite[Proposition 1.6 parts 4,5]{Ueno88}.  Next we give a multiplicative relationship on the polynomials $P_{\mu,\lambda}(y)$.  For any partition $\lambda=(\lambda_1\geq\cdots\geq \lambda_k)$ and $r\leq k$, define the subpartitions
$$\lambda(r):=(\lambda_1\geq\cdots\geq \lambda_r)\quad\text{and}\quad \lambda(r)^c:=(\lambda_{r+1}\geq\cdots\geq \lambda_k).$$

\begin{lem}\label{L:product}
Let $\mu\leq\lambda$ and suppose that $\mu_r=\lambda_r$ for some index $r$. Then $$P_{\mu,\lambda}(y)=P_{\mu(r),\lambda(r)}(y)\cdot P_{\mu(r)^c,\lambda(r)^c}(y).$$
\end{lem}

\begin{proof}
Consider the poset map
$$\phi:[\mu,\lambda]\rightarrow[\mu(r),\lambda(r)]\times[\mu(r)^c,\lambda(r)^c]$$
given by $\phi(\nu)=(\nu(r),\nu(r)^c)$.  It is easy see that $\phi$ is injective.  To show that $\phi$ is surjective, let $(\nu',\nu'')\in [\mu(r),\lambda(r)]\times[\mu(r)^c,\lambda(r)^c]$.  Then $\nu'_r=\mu_r=\lambda_r\geq\lambda_{r+1}\geq\nu''_1$ and hence $(\nu',\nu'')$ can be viewed as an element in $[\mu,\lambda]$.  The lemma now follows from the fact that $y^{|\mu|}=y^{|\mu(r)|}\cdot y^{|\mu(r)^c|}$.
\end{proof}

\begin{example}\label{example-product}
Let $\mu = (4, 3, 3,2,1).$ and $\lambda = (5, 4, 3, 3, 2,2)$.  Then $\mu_3=\lambda_3$ and Lemma \ref{L:product} implies
$$P_{(4, 3, 3,2,1),(5, 4, 3, 3, 2,2)}(y)=P_{(4, 3, 3),(5, 4, 3)}(y)\cdot P_{(2,1),(3, 2,2)}(y).$$
\[
\begin{tikzpicture}
\node (a2) at (2,0) {\begin{ytableau}
*(gray)& *(gray)& *(gray)&*(gray) &\\
*(gray)& *(gray)& *(gray)& \\
*(gray)& *(gray)& *(gray) \\
*(gray)&*(gray) &  \\
*(gray)& \\
{} &
\end{ytableau}};
\node (a2) at (3.25,0) {=};
\node (a2) at (4.5,0) {\begin{ytableau}
*(gray)& *(gray)& *(gray)&*(gray) &\\
*(gray)& *(gray)& *(gray)& \\
*(gray)& *(gray)& *(gray)
\end{ytableau} };
\node (a2) at (5.65,0) {$\times$};
\node (a2) at (6.5,0) {\begin{ytableau}
*(gray)&*(gray) &  \\
*(gray)& \\
{} &
\end{ytableau}};
\draw [dashed, thick] (1,0) -- (2.8,0);
\end{tikzpicture}
\]

\end{example}
Observe that Lemmas \ref{L:sum} and \ref{L:product}, along with the fact that $P_{\lambda,\lambda}(y)=y^{|\lambda|}$, can be used to recursively compute any polynomial $P_{\mu,\lambda}(y).$

To prove Theorem \ref{T:main1}, we refine the generating function $Q_k$ into partial sums over partitions where last part has a fixed value.  For $m\geq 1$, define the set $$\Lambda(k,m):=\{\lambda\in \Lambda(k)\ |\ \lambda_k=m\}$$ and the function
$$Q_{k,m}:=\sum_{\lambda\in \Lambda(k,m)} P_\lambda(y)\cdot \textbf{x}^{\lambda}.$$
It is easy to see that $\ds Q_k=\sum_{m\geq 1}Q_{k,m}.$ For any partition $\lambda=(\lambda_1\geq\cdots\geq\lambda_k)$, define $$\rho(\lambda):=(\lambda_1-1, \lambda_2-1,\ldots, \lambda_k-1).$$  For $m\geq 2$, the map $\rho$ is gives a bijection
$$\rho:\Lambda(k,m)\rightarrow\Lambda(k,m-1),$$ and for $m\geq 1$, the composition $\rho^m$ gives a bijection
$$\rho^m:\Lambda(k,m)\rightarrow \bigsqcup_{k'<k}\Lambda(k').$$  We now give a recursive formula for $Q_{k,1}.$

\begin{prop}\label{P:Q_k1}
Let $k\geq 1$.  Then $\ds Q_{k,1}=x_k\cdot Q_{k-1}+y^k p_k\cdot\left(\sum_{j=0}^{k-1} Q_j\right).$
\end{prop}

\begin{proof}
Let $\lambda\in\Lambda(k,1)$.  By Lemmas \ref{L:sum} and \ref{L:product}, we have $$P_\lambda(y)=P_{\lambda(k-1)}(y)+P_{(1)^k, \lambda}(y)=P_{\lambda(k-1)}(y)+y^k\cdot P_{\rho(\lambda)}(y).$$
Moreover, the maps $$\Lambda(k,1)\rightarrow \Lambda(k-1)\quad\text{and}\quad \Lambda(k,1)\rightarrow \bigsqcup_{k'<k}\Lambda(k')$$ given respectively by $\lambda\mapsto \lambda(k-1)$ and $\lambda\mapsto \rho(\lambda)$ are bijections.  Hence

\begin{align*}
Q_{k,1} &=\sum_{\lambda\in\Lambda(k,1)} P_\lambda(y)\, \textbf{x}^\lambda\\
        &=\sum_{\lambda\in\Lambda(k,1)} P_{\lambda(k-1)}(y)\, \textbf{x}^\lambda+\sum_{\lambda\in\Lambda(k,1)} y^k\cdot P_{\rho(\lambda)}(y)\, \textbf{x}^\lambda\\
        &=\sum_{\lambda'\in\Lambda(k-1)} x_k\cdot P_{\lambda'}(y)\, \textbf{x}^{\lambda'}+  \sum_{j=0}^{k-1}\left(\sum_{\lambda''\in\Lambda(j)}y^k p_k\cdot P_{\lambda''} (y)\, \textbf{x}^{\lambda''}\right)\\
        &=x_k\cdot Q_{k-1}+y^k p_k\left(\sum_{j=0}^{k-1} Q_j\right). \end{align*} \end{proof}

We now consider $Q_{k,m}$ for $m\geq 2.$
\begin{lem}\label{L:product_sum}
If $\lambda\in\Lambda(k)$, then $$\ds P_\lambda(y)=P_{\rho(\lambda)}(y)+P_{(\lambda_k)^k, \lambda}(y)+\sum_{r=1}^{k-1} P_{(\lambda_r)^r, \lambda(r)}(y)\cdot P_{\rho(\lambda(r)^c)}(y).$$
\end{lem}
\begin{proof}
For any $0\leq r\leq k$, define the partition $\mu^r=(\mu^r_1\geq \ldots\geq\mu^r_k)$ by
$$\mu^r_i:=\begin{cases}
\lambda_i &\text{if}\quad i\leq r\\
\lambda_i-1 & \text{if}\quad i> r.
\end{cases}$$
Note that for the extreme values  of $k$, we have $\mu^{k}=\lambda$ and $\mu^0=\rho(\lambda)$.  By construction, $\mu^r_{r}>\mu^r_{r+1}$ and hence Lemma \ref{L:sum} gives the recursive formula
$$P_{\mu^r}(y)=P_{\mu^{r-1}}(y)+P_{(\lambda_{r})^{r}, \mu^r}(y).$$
Since $\lambda=\mu^{k}$, we have
$$P_{\lambda}(y)=\sum_{r=0}^k P_{(\lambda_{r})^{r}, \mu^r}(y)=P_{\rho(\lambda)}(y)+P_{(\lambda_k)^k, \lambda}(y)+\sum_{r=1}^{k-1} P_{(\lambda_{r})^{r}, \mu^r}(y).$$
For each $1<r<k$, we have $\mu^r_r=\lambda_r$.  Applying Lemma \ref{L:product} gives
$$P_{(\lambda_{r})^{r}, \mu^r}(y)=P_{(\lambda_{r})^{r}, \mu^r(r)}(y)\cdot P_{\mu^r(r)^c}(y)=P_{(\lambda_{r})^{r}, \lambda(r)}(y)\cdot P_{\rho(\lambda(r)^c)}(y).$$
This proves the lemma.

\end{proof}
Lemma \ref{L:product_sum} implies
\begin{align*}
Q_{k,m}=&\sum_{\lambda\in \Lambda(k,m)} P_\lambda(y)\cdot \textbf{x}^{\lambda}\\
=&\sum_{\lambda\in \Lambda(k,m)} P_{\rho(\lambda)}(y)\cdot\textbf{x}^{\lambda} \\
&+\sum_{\lambda\in \Lambda(k,m)} P_{(\lambda_k)^k, \lambda}(y)\cdot\textbf{x}^{\lambda}\\
&+ \sum_{\lambda\in \Lambda(k,m)}\left(\sum_{r=1}^{k-1} P_{(\lambda_r)^r, \lambda(r)}(y)\cdot P_{\rho(\lambda(r)^c)}(y)\right) \cdot \textbf{x}^{\lambda}\\
\end{align*}
We consider the three summands above separately.  For the first summand, recall that for $m\geq 2$, the map $\rho:\Lambda(k,m)\rightarrow\Lambda(k,m-1)$ is a bijection and $\textbf{x}^\lambda=p_k\cdot\textbf{x}^{\rho(\lambda)}$.  Hence
\begin{equation}\label{Eq:1st_summand}
\sum_{\lambda\in \Lambda(k,m)} P_{\rho(\lambda)}(y)\cdot \textbf{x}^{\lambda}=\sum_{\lambda\in \Lambda(k,m)} P_{\rho(\lambda)}(y)\cdot p_k\cdot\textbf{x}^{\rho(\lambda)}=p_k\cdot Q_{k,m-1}.
\end{equation}
For the second summand, recall that $\rho^m:\Lambda(k,m)\rightarrow \bigsqcup_{i<k}\Lambda(i)$ is also a bijection and $\textbf{x}^\lambda=p_k^m\cdot \textbf{x}^{\rho^m(\lambda)}$.  Moreover, by Lemma \ref{L:product}, $P_{(\lambda_k)^k, \lambda}(y)=y^{mk}\cdot P_{\rho^m(\lambda)}(y)$ and hence
\begin{equation}\label{Eq:2nd_summand}
\sum_{\lambda\in \Lambda(k,m)} P_{(\lambda_k)^k, \lambda}(y)\cdot\textbf{x}^{\lambda}=\sum_{\lambda\in \Lambda(k,m)} \left(y^{mk}\cdot P_{\rho^m(\lambda)}(y)\right)\cdot \left(p_k^m\cdot \textbf{x}^{\rho^m(\lambda)}\right)=(y^k p_k)^m\sum_{i=0}^{k-1}Q_{i}.
\end{equation}
Finally, for the third summand, we switch the order of summation and focus on expressions of the form
$$\sum_{\lambda\in \Lambda(k,m)} P_{(\lambda_r)^r, \lambda(r)}(y)\cdot P_{\rho(\lambda(r)^c)}(y)\cdot \textbf{x}^{\lambda}$$
for fixed values of $r$.
\begin{prop}\label{P:3rd_summand}
Let $m\geq 2.$  If $1\leq r< k$, then $$\sum_{\lambda\in \Lambda(k,m)} P_{(\lambda_r)^r, \lambda(r)}(y)\cdot P_{\rho(\lambda(r)^c)}(y)\cdot \textbf{\emph{x}}^{\lambda}=\left(\frac{y^r p_k}{1-y^r p_r}\right)\left(\sum_{i=0}^{r-1} Q_i\right) Q_{k-r, m-1}(Z_r)$$
where $Z_r:= \left(y^rp_{r+1},x_{r+2},\cdots,x_k,y \right).$
\end{prop}
\begin{proof}

First, note that
$$\sum_{\lambda\in \Lambda(k,m)} P_{(\lambda_r)^r, \lambda(r)}(y)\cdot P_{\rho(\lambda(r)^c)}(y)\cdot \textbf{x}^{\lambda}= \sum_{t\geq m}\left(\sum_{\substack{\lambda\in \Lambda(k,m)\\ \lambda_r=t}} P_{(t)^r, \lambda(r)}(y)\cdot P_{\rho(\lambda(r)^c)}(y)\cdot \textbf{x}^{\lambda}\right)$$
For $t\geq m$, define the sets $$\Lambda(r,k;t,m):=\{\lambda\in \Lambda(k,m)\ |\ \lambda_r=t\},$$
$$\bar\Lambda(r):=\{\lambda\in \Lambda(r')\ |\ r'<r\},\quad\text{and}\quad\Lambda_{t}(k,m):=\{\lambda\in \Lambda(k,m)\ |\ \lambda_1<t\}.$$
Consider the map given by $\Lambda(r,k;t,m)\rightarrow \bar\Lambda(r)\times\Lambda_t(k-r, m-1)$ given by $$\lambda\mapsto (\rho^t(\lambda(r)), \rho(\lambda(r)^c)).$$
In other words,
$$\rho^t(\lambda(r))=(\lambda_1-t,\lambda_2-t,\ldots, \lambda_r-t)\quad\text{and}\quad\rho(\lambda(r)^c)=(\lambda_{r+1}-1,\ldots,\lambda_k-1).$$
For example, if $\lambda=(10, 8, 7,5,4)$ and $r=3$, then $t=7$ and
$$(10, 8, 7,5,4)\mapsto ((3,1,0),(4,3))$$
$$
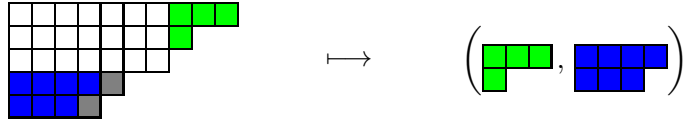
\begin{tikzpicture}
\node (a1) at (0,0) { \begin{ytableau}
*(white)& *(white)& *(white)& *(white)& *(white)& *(white)& *(white)&  *(green)& *(green)& *(green)\\
*(white)& *(white)& *(white)& *(white)& *(white)& *(white)& *(white)&  *(green)  \\
*(white)& *(white)& *(white)& *(white)& *(white)& *(white)& *(white)  \\
*(blue)& *(blue) & *(blue) & *(blue) & *(gray)\\
*(blue) & *(blue)& *(blue) & *(gray)
\end{ytableau}};
\node (a3) at (3,0) {$\longmapsto $};
\node (a2) at (6,0) {$ \left(
\begin{ytableau}
*(green)& *(green)& *(green)\\
*(green)
\end{ytableau}\,,\,
\begin{ytableau}
*(blue)& *(blue) & *(blue) & *(blue) \\
*(blue) & *(blue)& *(blue)
\end{ytableau}\right)  $};
\end{tikzpicture}
$$
It is easy to see that this map is a bijection and for notational simplicity, define $\lambda':=\rho^t(\lambda(r))$ and $\lambda'':=\rho(\lambda(r)^c)$ for $\lambda\in\Lambda(r,k;t,m)$.  Observe that the monomial $\textbf{x}^\lambda$ factors as
$$\textbf{x}^\lambda=\textbf{x}^{\lambda'}\cdot(p_r)^{t}\cdot(x_{r+1}^{\lambda_{r+1}}\cdots x_{k}^{\lambda_{k}}).$$
We now have
\begin{align*}
\sum_{\lambda\in \Lambda(r,k; t, m)} &P_{(t)^r, \lambda(r)}(y) \cdot P_{\rho(\lambda(r)^c)}(y)\cdot \textbf{x}^{\lambda}\\
&=\sum_{(\lambda',\lambda'')\in \bar\Lambda(r)\times \Lambda_t(k-r,m-1)}y^{rt}\cdot P_{\lambda'}(y)\cdot P_{\lambda''}(y)\cdot\textbf{x}^{\lambda'}\cdot(p_r)^{t}\cdot(x_{r+1}^{\lambda_{r+1}}\cdots x_{k}^{\lambda_{k}})\\
&= \sum_{\lambda'\in \bar\Lambda(r)} P_{\lambda'}(y)\cdot\textbf{x}^{\lambda'}\cdot\left( \sum_{\lambda''\in \Lambda_t(k-r,m-1)}(y^{r}p_r)^{t}\cdot P_{\lambda''}(y)\cdot (x_{r+1}^{\lambda_{r+1}}\cdots x_{k}^{\lambda_{k}})\right)\\
&= \left(\sum_{i=0}^{r-1} Q_i\right)\cdot\left( \sum_{\lambda''\in \Lambda_t(k-r,m-1)}(y^{r}p_r)^{t}\cdot P_{\lambda''}(y)\cdot (x_{r+1}^{\lambda_{r+1}}\cdots x_{k}^{\lambda_{k}})\right).
\end{align*}
Summing the second factor above over $t\geq m$ gives:

\begin{align*}
\sum_{t\geq m}\Bigg(\sum_{\lambda''\in \Lambda_t(k-r,m-1)}&(y^{r}p_r)^{t}\cdot P_{\lambda''}(y)\cdot (x_{r+1}^{\lambda_{r+1}}\cdots x_{k}^{\lambda_{k}})\Bigg)\\
&= \sum_{\lambda''\in \Lambda(k-r,m-1)}\left(\sum_{t> \lambda''_1} (y^{r}p_r)^{t}\cdot P_{\lambda''}(y)\cdot (x_{r+1}^{\lambda_{r+1}}\cdots x_{k}^{\lambda_{k}})\right)\\
&= \left(\frac{1}{1-y^rp_r}\right)\sum_{\lambda''\in \Lambda(k-r,m-1)}(y^{r}p_r)^{\lambda''_1+1}\cdot P_{\lambda''}(y)\cdot (x_{r+1}^{\lambda_{r+1}}\cdots x_{k}^{\lambda_{k}})\\
&= \left(\frac{y^rp_k}{1-y^rp_r}\right)\sum_{\lambda''\in \Lambda(k-r,m-1)} P_{\lambda''}(y)\cdot (y^{r}p_{r+1})^{\lambda''_{1}}\cdot x_{r+2}^{\lambda''_2}\cdots x_{k}^{\lambda''_{k-r}}\\
&= \left(\frac{y^rp_k}{1-y^rp_r}\right)Q_{k-r,m-1}(Z_r)
\end{align*}

where $Z_r=(y^{r}p_{r+1}, x_{r+2},\ldots,x_k,y)$.  Combining the two calculations above proves the lemma.\end{proof}

We now prove our first main theorem.

\begin{proof}[Proof of Theorem \ref{T:main1}]
By Proposition \ref{P:Q_k1}, we have
$$\ds Q_{k,1}=x_k\cdot Q_{k-1}+y^k p_k\left(\sum_{i=0}^{k-1} Q_i\right).$$

For $m\geq 2$, Equations \eqref{Eq:1st_summand} and \eqref{Eq:2nd_summand} and Proposition \ref{P:3rd_summand} imply
$$Q_{k,m}=p_k\cdot Q_{k,m-1}+(y^k p_k)^m\left(\sum_{i=0}^{k-1} Q_i\right) + \sum_{r=1}^{k-1}\left(\frac{y^r p_k}{1-y^r p_r}\right)\left(\sum_{i=0}^{r-1} Q_i\right) Q_{k-r, m-1}(Z_r).$$
Combining these results we get

\begin{align*}
Q_k&=\sum_{m=1}^{\infty} Q_{k,m}=Q_{k,1}+\sum_{m=2}^{\infty} Q_{k,m}\\
&=x_k\cdot Q_{k-1}+p_k \sum_{m=2}^{\infty}Q_{k,m-1}+(y^k p_k)\left(\sum_{i=0}^{k-1} Q_i\right)+\sum_{m=2}^{\infty}(y^k p_k)^m\left(\sum_{i=0}^{k-1} Q_i\right)\\
&\hspace{2in}+\sum_{m=2}^{\infty}\sum_{r=1}^{k-1}\left(\frac{y^r p_k}{1-y^r p_r}\right)\left(\sum_{i=0}^{r-1} Q_i\right) Q_{k-r, m-1}(Z_r)\\
&=x_k\cdot Q_{k-1}+p_k\cdot Q_k+\left(\frac{y^k p_k}{1-y^k p_k}\right)\left(\sum_{i=0}^{r-1} Q_i\right)\\
&\hspace{2in}+\sum_{r=1}^{k-1}\left(\frac{y^r p_k}{1-y^r p_r}\right)\left(\sum_{i=0}^{r-1} Q_i\right)\left( \sum_{m=2}^{\infty} Q_{k-r, m-1}(Z_r)\right)\\ \end{align*}

\begin{align*}
&=x_k\cdot Q_{k-1}+p_k\cdot Q_k+\sum_{r=1}^k\left(\frac{y^r p_k}{1-y^r p_r}\right)\left(\sum_{i=0}^{r-1} Q_i\right)Q_{k-r}(Z_r).
\end{align*}
\end{proof}

\section{Asymptotics of partition intervals}
In this section we study the asymptotic behavior of various degenerations of the generating series $Q_k$ and prove Theorem \ref{T:main2}.  For any $i<r$, define
\begin{equation*}\label{Eq:Main_theorem_summands_Rir}
R_{i,r}=R_{i,r}(x_1,\ldots,x_k,y):=\frac{(y^rp_k)\cdot Q_i\cdot Q_{k-r}(Z_r)}{(1-p_k)(1-y^rp_r)},
\end{equation*}
where $Z_r=(p_{r+1}y^r, x_{r+2},\ldots,x_k,y)$.  Theorem \ref{T:main1} states that
\begin{equation}\label{Eq:Main_theorem_summands}
Q_k=\frac{x_kQ_{k-1}}{(1-p_k)}+\sum_{0\leq i<r\leq k} R_{i,r}.
\end{equation}
We first study the singularities of $Q_k.$  Define the polynomial
$$D_k=D_k(x_1,\ldots,x_k,y):=\prod_{m=1}^k\prod_{j=0}^m(1-y^j p_m).$$
\begin{prop}\label{P:denominator}
The product $Q_k\cdot D_k$ is a polynomial.
\end{prop}
\begin{proof}
We prove the proposition by induction on $k$.  When $k=1$, we have
$$Q_1(x_1,y)=\frac{x_1+(x_1-x_1^2)y}{(1-x_1)(1-x_1y)}=\frac{x_1+(x_1-x_1^2)y}{D_1(x_1,y)}.$$
Suppose the lemma is true for all $i<k.$  We study the summands given in Equation \eqref{Eq:Main_theorem_summands}.  First note that $$D_k=D_{k-1}\cdot \prod_{j=0}^k (1-y^jp_k).$$  By induction, $\ds\frac{x_kQ_{k-1}}{1-p_k}\cdot D_k$ is a polynomial.  For the summand $R_{i,r}$, we write
$$D_i=\prod_{m=1}^i\prod_{j=0}^m(1-y^j p_m)\quad\text{and}\quad D_{k-r}(Z_r)=\prod_{m=r+1}^k\prod_{j=r}^m(1-y^{j} p_m).$$ Hence for any $i<r\leq k$, we have
$$D_k=D_i\cdot D_{k-r}(Z_r)\left(\prod_{m=i+1}^r\prod_{j=0}^m(1-y^j p_m)\right)\left(\prod_{m=r+1}^k\prod_{j=0}^{r-1}(1-y^j p_m)\right).$$
In particular, the product $D_i\cdot D_{k-r}(Z_r)\cdot (1-p_k)\cdot (1-y^rp_r)$ divides $D_k.$
By induction, $R_{i,r}\cdot D_k$ is a polynomial for all $i<r\leq k$.  Thus $Q_k\cdot D_k$ is also a polynomial.
\end{proof}
The following lemma is an elementary fact about rational functions in one variable and we leave the proof as an exercise.
\begin{lem}\label{L:pole_sums}
Let $f_1(x),\ldots, f_s(x)$ be a collection of rational functions with poles of order $p_1,\ldots,p_s$ at $x=a$ respectively.  Let $g_i(x):=(a-x)^{p_i}\cdot f_i(x)$ for each $1\leq i\leq s$.  Define $$\ds F(x):=\sum_{i=1}^s f_i(x)\quad\text{and}\quad P:=\max\{p_1,\ldots,p_s\}.$$  Then the following are true:
\begin{enumerate}
\item $F(x)$ has a pole of order at most $P$ at $x=a$.
\item If $g_i(a)>0$ for all $i$, then $F(x)$ has a pole of exactly order $P$ at $x=a$.  Furthermore, if $G(x)=(a-x)^P\cdot F(x)$, then $G(a)>0.$
\end{enumerate}
\end{lem}

For any $m\geq 0$, define the set of variables $X_m:=(x^{m+1},x\ldots,x,1)$ and consider the single variable function
$$Q_k(X_m)=Q_k(x^{m+1},x,\ldots,x,1).$$
Proposition \ref{P:denominator} implies that the singularities of $Q_k(X_m)$ are roots of unity and by Theorem \ref{T:main1}, the functions $Q_k(X_m)$ satisfy the relation:
\begin{equation}\label{Eq:single_variable_recursion}
Q_k(X_m)=\frac{x\cdot Q_{k-1}(X_m)}{(1-x^{k+m})}+\sum_{0\leq i<r\leq k}\left(\frac{x^{k+m}\cdot Q_{k-r}(X_{r+m})\cdot Q_{i}(X_m)}{(1-x^{k+m})(1-x^{r+m})}\right).
\end{equation}
We denote the summands in later part of Equation \eqref{Eq:single_variable_recursion} by
$$R_{i,r}(X_m)=R_{i,r}(x^{m+1},x,\ldots,x,1)=\frac{x^{k+m}\cdot Q_{k-r}(X_{r+m})\cdot Q_{i}(X_m)}{(1-x^{k+m})(1-x^{r+m})}.$$

\begin{prop}\label{P:Q_k^m_singularities}
The following are true:
\begin{enumerate}
\item The function $Q_k(X_m)$ has a pole of order $2k$ at $x=1.$
\item If $k\geq 2$ and $x_0\neq 1$ is a singularity of $Q_k(X_m)$, then the order of $x_0$ is strictly less than $2k$.
\end{enumerate}
\end{prop}

\begin{proof}
We prove both parts of the proposition by induction on $k$. For $k=1$, we have
$$Q_1(X_m)=\frac{2x^{m+1}-x^{2m+2}}{(1-x^{m+1})^2}$$
which has a pole of order 2 at $x=1$ (the order is 2 at all $(m+1)$-th roots of unity).  Define $g_{k,m}(x):=(1-x)^{2k}\cdot Q_k(X_m).$ Note for $k=1$, we have $g_{1,m}(1)>0.$  Suppose for the sake of induction that $Q_l(X_m)$ has a pole of order $2l$ at $x=1$ and $g_{l,m}(1)>0$ for all $l<k$ and $m\geq 0$.  By induction, $R_{i,r}(X_m)$ has a pole of order $2k-2r+2i+2$ at $x=1$.  Lemma \ref{L:pole_sums} part (2) and Equation \eqref{Eq:single_variable_recursion} imply that $Q_k(X_m)$ has a pole of order $2k=\max\{2k-1, 2k-2r+2i+2\ |\ 0\leq i<r\leq k\}$ at $x=1$.  In particular, the maximum is achieved whenever $i=r-1$.  Lemma \ref{L:pole_sums} part (2) further implies that $g_{k,m}(1)>0$ which proves the first part of the proposition.

To prove part (2) of the proposition, note that for $k=2$,
$$ Q_2(X_m)=\frac{3x^{m+2}-2x^{2m+3}-2x^{2m+4}+x^{3m+6}}{(1-x)^2(1-x^{m+2})^3}$$
which has singularities of at most order 3 away from $x=1$.  Let $x_0\neq 1$ be a singularity of $Q_k(X_m)$.  For the sake of induction, suppose that for all $2\leq l<k$ and $m\geq 0$, the order of $x_0$ for $Q_{l,m}$ is strictly less than $2l$.  We now look at the summands given in Equation \eqref{Eq:single_variable_recursion}.  By induction, the first summand $\ds\frac{x\cdot Q_{k-1}(X_m)}{(1-x^{k+m})}$ has a pole of order at most $2k-1$ at $x_0$.  For the summands $R_{i,r}(X_m)$, we consider two cases.  First suppose that $k-r\geq 2$.   By induction, the order of $x_0$ for $R_{i,r}(X_m)$ is at most $(2k-2r-1)+2i+2\leq 2k-1$.  Now suppose that $k-r=1$ and hence $k+m$, $r+m$ are relatively prime.  In this case, the factors $(1-x^{k+m})$ and $(1-x^{r+m})$ appearing in the denominator of $R_{i,r}(X_m)$ share no common roots except $x=1$.  Since $x_0\neq 1$, the order of $x_0$ is at most $(2k-2r)+2i+1\leq 2k-1$. By Lemma \ref{L:pole_sums} part (1) and Equation \eqref{Eq:single_variable_recursion}, the order of $x_0$ for $Q_k(X_m)$ is strictly less than $2k.$  This proves part two of the proposition.
\end{proof}

Define the coefficients $C_{k,n}^m$ by the expansion
$$Q_k(X_m)=\sum_{n=k}^\infty C_{k,n}^m\, x^n.$$
When $m=0$, the coefficient $$\ds C_{k,n}:=C_{k,n}^0=\sum_{\substack{\lambda\in \Lambda(k)\\ |\lambda|=n}}\#[\emptyset, \lambda].$$  Proposition \ref{P:Q_k^m_singularities} and \cite[Theorem IV.9 and Exercise IV.26]{Flajolet-Sedgewick09} imply that the asymptotic growth of the coefficients $C_{k,n}^m$, as $n\rightarrow\infty$, is controlled by the singularity of $Q_k^m(X_m)$ at $x=1$.  In particular, we have
$$C_{k,n}^m\sim \frac{B(k,m)\cdot n^{2k-1}}{(2k-1)!}\quad\text{where}\quad B(k,m):=\lim_{x\rightarrow 1}(1-x)^{2k}\cdot Q_k(X_m).$$
The constants $B(k,m)$ can be computed recursively using the following proposition.

\begin{prop}\label{P:Bkm_recursion}
The constants $B(k,m)$ satisfy the recursion
\begin{equation}\label{Eq:B_km_recursion}
B(k,m)=\sum_{r=1}^k\frac{B(k-r,r+m)B(r-1,m)}{(m+k)(m+r)}
\end{equation}
where $B(0,m)=1$ for all $m\geq 0$.
\end{prop}

\begin{proof}
For $k=0$, we have $Q_0(X_m)=1$ and hence $B(0,m)=1$.  For $k\geq 1$, we consider the recursive formula for $Q_k(X_m)$ given in Equation \eqref{Eq:single_variable_recursion}.  The summand $\ds\frac{x\cdot Q_{k-1}(X_m)}{(1-x^{k+m})}$ has a pole of order $2k-1$ at $x=1$ and hence
$$\lim_{x\rightarrow 1}(1-x)^{2k}\cdot\frac{x\cdot \tilde Q_{k-1}(X_m)}{(1-x^{k+m})}=0.$$
Proposition \ref{P:Q_k^m_singularities} part (1) implies the summands $R_{i,r}(X_m)$ have a pole of order at most $2k$ at $x=1$, with equality if and only if $r=i+1$.  Thus

\begin{align*}
\lim_{x\rightarrow 1}(1-x)^{2k}\cdot R_{r-1,r}(X_m)&=\lim_{x\rightarrow 1}\frac{(1-x)^{2k}\cdot Q_{k-r}(X_{r+m})\cdot Q_{r-1}(X_m)}{(1-x^{k+m})(1-x^{r+m})}\\
&=\lim_{x\rightarrow 1}\frac{(1-x)^{2}\cdot B(k-r,r+m) B(r-1,m)}{(1-x^{k+m})(1-x^{r+m})}\\
&=\frac{B(k-r,r+m) B(r-1,m)}{(k+m)(r+m)}
\end{align*}
and
$$\ds \lim_{x\rightarrow 1}(1-x)^{2k}\cdot R_{i,r}^m(X_m)=\begin{cases}\ds \frac{B(k-r,r+m) B(r-1,m)}{(k+m)(r+m)} & \text{ if }   r=i+1 \\ & \\ \hspace{0.75in}
0 & \text{ if }   r>i+1.\end{cases}$$
Applying Equation \eqref{Eq:single_variable_recursion} proves the proposition.
\end{proof}

Recall from Section \ref{S:main_results} that we define
$$A_{k,n}:=\frac{C_{k,n}}{c_{k,n}}$$
where $c_{k,n}$ denotes the number of partitions in $\Lambda(k)$ of rank $n.$  Applying \cite[Theorem IV.9 and Example IV.6]{Flajolet-Sedgewick09} to the generating function for $c_{k,n}$, we get
$$c_{k,n}\sim \frac{n^{k-1}}{k!(k-1)!}.$$
Thus
\begin{equation}\label{Eq:A_km_calculation}
\ds A_{k,n}=\frac{C_{k,n}}{c_{k,n}}\sim \left(\frac{B(k,0)\cdot n^{2k-1}}{(2k-1)!}\right)\cdot \left(\frac{k!(k-1)!}{n^{k-1}}\right)=\frac{k!(k-1)!}{(2k-1)!}\cdot B(k,0)\cdot n^k.
\end{equation}
To prove Theorem \ref{T:main2}, we need to show that Equation \eqref{Eq:A_km_calculation} is consistent with the growth constant $G_k$ defined in Section \ref{S:main_results}.  For any $p,q\in\Z$, let $$(p)_q:=p(p-1)\cdots(p-q+1)$$ denote the corresponding falling factorial.  For $k\geq 1$ and $m\geq 0$, define
$$\widetilde B(k, m):=\frac{1}{(m+k)_k}\cdot\sum_{\lambda\in I_{k,m}}\left(\prod_{i=1}^k\,\lambda_i^{-1}\right)$$
where the interval $$I_{k,m}:=[(m+k,m+k-1,m+k-2,\ldots,m+1), (m+k)^k].
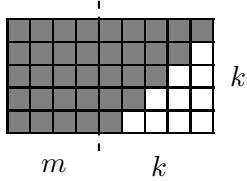
$$
$$
\begin{tikzpicture}
\node (a1) at (0,0) { \begin{ytableau}
*(gray)&*(gray)&*(gray)&*(gray)&*(gray)& *(gray)& *(gray)& *(gray) & *(gray) \\
*(gray)&*(gray)&*(gray)&*(gray)&*(gray)& *(gray)& *(gray)& *(gray) &  \\
*(gray)&*(gray)&*(gray)&*(gray)&*(gray)& *(gray)& *(gray)& *(white) & \\
*(gray)&*(gray)&*(gray)&*(gray)&*(gray)& *(gray)& *(white)& *(white) & \\
*(gray)&*(gray)&*(gray)&*(gray)&*(gray)& *(white)& *(white)& *(white) &
\end{ytableau}};
\draw [dashed, thick] (-0.15,1) -- (-0.15,-1);
\node (a2) at (-0.75,-1.2) {$m$};
\node (a3) at (0.65,-1.2) {$k$};
\node (a4) at (1.7,0) {$k$};
\end{tikzpicture}
$$
When $k=0$, we set $\widetilde B(0, m):=1$.  We prove Theorem \ref{T:main2} by showing $B(k,0)=\widetilde B(k,0)$ using the recursive formula in Proposition \ref{P:Bkm_recursion}.  The following lemma is analogus to Lemma \ref{L:sum}, only we add a box to $\mu$, instead of deleting a box from $\lambda.$

\begin{lem}\label{L:sum2} Let $\lambda\in\Lambda(k)$ and $\mu\leq\lambda.$  Let $r>0$ be an index such that $\mu_r<\mu_{r-1}$ and $\mu_r<\lambda_r$. Define the partitions $\bar \mu$ and $\bar\lambda$ by

$$
  \bar{\mu}_i :=
  \begin{cases}
                                   \mu_i & \text{ if }  i\neq r \\
                                   \mu_i+1 & \text{ if }  i=r
  \end{cases}  \qquad   \text{and}\qquad
  \bar{\lambda}_i :=
  \begin{cases}
                                   \lambda_i& \text{if } 1\leq i< r \\
                                   \mu_r & \text{if } r\leq i\leq k
    \end{cases}.$$
Then the interval
$$[\mu,\lambda]=[\bar\mu,\lambda]\sqcup[\mu,\bar\lambda].$$
\end{lem}

\begin{proof}
Let $\nu\in[\mu,\lambda]$.  If $\nu_r>\mu_r$, then $\nu\geq \bar\mu$ and hence $\nu\in[\bar\mu,\lambda]$.  Otherwise, $\nu_r=\mu_r$ and $\nu\leq \bar\lambda.$  In this case, $\nu\leq \bar\lambda$ and $\nu\in [\mu,\bar\lambda]$.
\end{proof}

Next, we apply repeated applications of Lemma \ref{L:sum2} to the interval $I_{k,m}.$  For $r\leq k$ and $m\geq 0$, define the partitions $\alpha_{k,m}(r),\beta_{k,m}(r)\in\Lambda(k)$ by

$$
  \alpha_{k,m}(r)_i :=
  \begin{cases}
                                   m+k & \text{ if }  i=1 \\
                                   m+k+1-i & \text{ if }  i> k-r+1\\
                                   m+k+2-i & \text{ if }  2\leq i\leq k-r+1 \\

  \end{cases}
$$
and
$$
  \beta_{k,m}(r)_i :=
  \begin{cases}
                                   m+k& \text{if } i\leq k-r+1 \\
                                   m+r-1 & \text{if } i> k-r+1
    \end{cases}.$$

By definition, we have $I_{k,m}=[\alpha_{k,m}(k),\beta_{k,m}(1)]$ and by Lemma \ref{L:sum2}, $$[\alpha_{k,m}(r),\beta_{k,m}(1)]=[\alpha_{k,m}(r-1),\beta_{k,m}(1)]\sqcup[\alpha_{k,m}(r),\beta_{k,m}(r)].$$
This implies
\begin{equation}\label{Eq:Ikm_decomp}
I_{k,m}=\bigsqcup_{1\leq r\leq k} [\alpha_{k,m}(r),\beta_{k,m}(r)].
\end{equation}
For example, when $k=5$ and $m=1$, diagrammatically applying Lemma \ref{L:sum2} to $I_{5,1}$ gives
$$
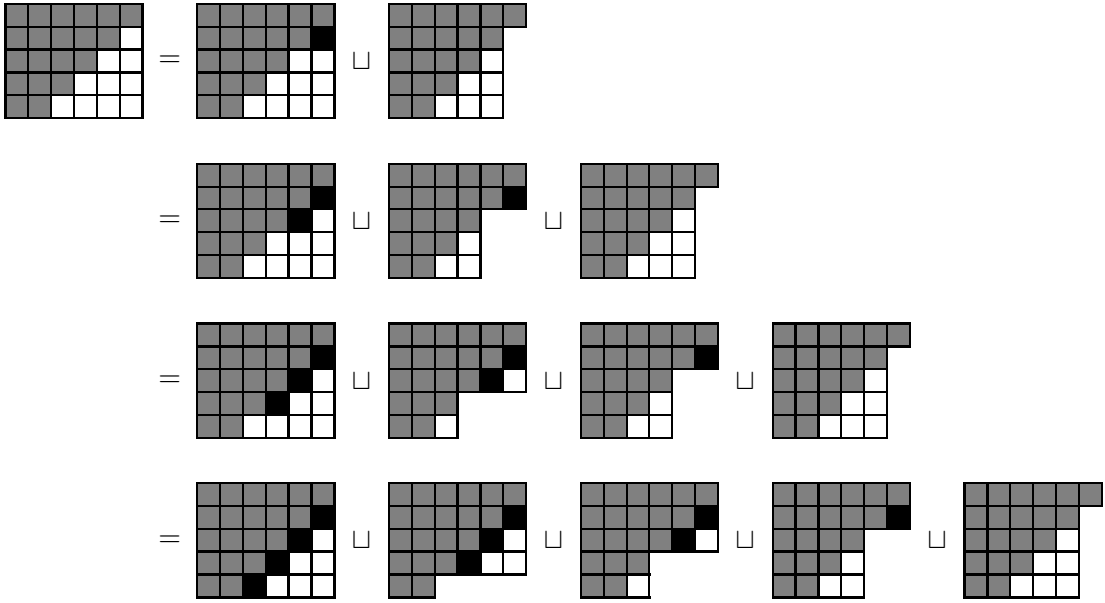
\begin{tikzpicture}[scale=0.85]
\node (a1) at (0,0) { \begin{ytableau}
*(gray)&*(gray)& *(gray)& *(gray)& *(gray) & *(gray) \\
*(gray)&*(gray)& *(gray)& *(gray)& *(gray) & *(white) \\
*(gray)&*(gray)& *(gray)& *(gray)& *(white) & \\
*(gray)&*(gray)& *(gray)& *(white)& *(white) & \\
*(gray)&*(gray)& *(white)& *(white)& *(white) &
\end{ytableau}};
\node (a2) at (1.5,0) {$ = $};
\node (a3) at (3,0) {
\begin{ytableau}
*(gray)&*(gray)& *(gray)& *(gray)& *(gray) & *(gray) \\
*(gray)&*(gray)& *(gray)& *(gray)& *(gray) & *(black) \\
*(gray)&*(gray)& *(gray)& *(gray)& *(white) & \\
*(gray)&*(gray)& *(gray)& *(white)& *(white) & \\
*(gray)&*(gray)& *(white)& *(white)& *(white) &
\end{ytableau} };
\node (a4) at (4.5, 0) {$\sqcup$};
\node (a5) at (6,0) {
\begin{ytableau}
*(gray)&*(gray)& *(gray)& *(gray)& *(gray) & *(gray) \\
*(gray)&*(gray)& *(gray)& *(gray)& *(gray)  \\
*(gray)&*(gray)& *(gray)& *(gray)&  \\
*(gray)&*(gray)& *(gray)&  & \\
*(gray)&*(gray)& & &
\end{ytableau}   };
\node (a2) at (1.5,-2.5) {$ = $};
\node (a3) at (3,-2.5) {
\begin{ytableau}
*(gray)&*(gray)& *(gray)& *(gray)& *(gray) & *(gray) \\
*(gray)&*(gray)& *(gray)& *(gray)& *(gray) & *(black) \\
*(gray)&*(gray)& *(gray)& *(gray)& *(black) & \\
*(gray)&*(gray)& *(gray)& *(white)& *(white) & \\
*(gray)&*(gray)& *(white)& *(white)& *(white) &
\end{ytableau} };
\node (a4) at (4.5, -2.5) {$\sqcup$};
\node (a5) at (6,-2.5) {
\begin{ytableau}
*(gray)&*(gray)& *(gray)& *(gray)& *(gray) & *(gray) \\
*(gray)&*(gray)& *(gray)& *(gray)& *(gray) & *(black) \\
*(gray)&*(gray)& *(gray)& *(gray) \\
*(gray)&*(gray)& *(gray) &  \\
*(gray)&*(gray)& &
\end{ytableau}   };
\node (a4) at (7.5, -2.5) {$\sqcup$};
\node (a5) at (9,-2.5) {
\begin{ytableau}
*(gray)&*(gray)& *(gray)& *(gray)& *(gray) & *(gray) \\
*(gray)&*(gray)& *(gray)& *(gray)& *(gray)  \\
*(gray)&*(gray)& *(gray)& *(gray)&  \\
*(gray)&*(gray)& *(gray)&  & \\
*(gray)&*(gray)& & &
\end{ytableau}   };
\node (a2) at (1.5,-5) {$ = $};
\node (a3) at (3,-5) {
\begin{ytableau}
*(gray)&*(gray)& *(gray)& *(gray)& *(gray) & *(gray) \\
*(gray)&*(gray)& *(gray)& *(gray)& *(gray) & *(black) \\
*(gray)&*(gray)& *(gray)& *(gray)& *(black) & \\
*(gray)&*(gray)& *(gray)& *(black)& *(white) & \\
*(gray)&*(gray)& *(white)& *(white)& *(white) &
\end{ytableau} };
\node (a4) at (4.5, -5) {$\sqcup$};
\node (a5) at (6,-5) {
\begin{ytableau}
*(gray)&*(gray)& *(gray)& *(gray)& *(gray) & *(gray) \\
*(gray)&*(gray)& *(gray)& *(gray)& *(gray) & *(black) \\
*(gray)&*(gray)& *(gray)& *(gray)& *(black) & \\
*(gray)&*(gray)& *(gray) \\
*(gray)&*(gray) &
\end{ytableau}   };
\node (a4) at (7.5, -5) {$\sqcup$};
\node (a5) at (9,-5) {
\begin{ytableau}
*(gray)&*(gray)& *(gray)& *(gray)& *(gray) & *(gray) \\
*(gray)&*(gray)& *(gray)& *(gray)& *(gray) & *(black) \\
*(gray)&*(gray)& *(gray)& *(gray) \\
*(gray)&*(gray)& *(gray) &  \\
*(gray)&*(gray)& &
\end{ytableau}  };
\node (a4) at (10.5, -5) {$\sqcup$};
\node (a5) at (12,-5) {
\begin{ytableau}
*(gray)&*(gray)& *(gray)& *(gray)& *(gray) & *(gray) \\
*(gray)&*(gray)& *(gray)& *(gray)& *(gray)  \\
*(gray)&*(gray)& *(gray)& *(gray)&  \\
*(gray)&*(gray)& *(gray)&  & \\
*(gray)&*(gray)& & &
\end{ytableau} };
\node (a2) at (1.5,-7.5) {$ = $};
\node (a3) at (3,-7.5) {
\begin{ytableau}
*(gray)&*(gray)& *(gray)& *(gray)& *(gray) & *(gray) \\
*(gray)&*(gray)& *(gray)& *(gray)& *(gray) & *(black) \\
*(gray)&*(gray)& *(gray)& *(gray)& *(black) & \\
*(gray)&*(gray)& *(gray)& *(black)& *(white) & \\
*(gray)&*(gray)& *(black)& *(white)& *(white) &
\end{ytableau} };
\node (a4) at (4.5, -7.5) {$\sqcup$};
\node (a5) at (6,-7.5) {
\begin{ytableau}
*(gray)&*(gray)& *(gray)& *(gray)& *(gray) & *(gray) \\
*(gray)&*(gray)& *(gray)& *(gray)& *(gray) & *(black) \\
*(gray)&*(gray)& *(gray)& *(gray)& *(black) & \\
*(gray)&*(gray)& *(gray)& *(black)& *(white) & \\
*(gray)&*(gray)
\end{ytableau}   };
\node (a4) at (7.5, -7.5) {$\sqcup$};
\node (a5) at (9,-7.5) {
\begin{ytableau}
*(gray)&*(gray)& *(gray)& *(gray)& *(gray) & *(gray) \\
*(gray)&*(gray)& *(gray)& *(gray)& *(gray) & *(black) \\
*(gray)&*(gray)& *(gray)& *(gray)& *(black) & \\
*(gray)&*(gray)& *(gray) \\
*(gray)&*(gray) &
\end{ytableau}  };
\node (a4) at (10.5, -7.5) {$\sqcup$};
\node (a5) at (12,-7.5) {
\begin{ytableau}
*(gray)&*(gray)& *(gray)& *(gray)& *(gray) & *(gray) \\
*(gray)&*(gray)& *(gray)& *(gray)& *(gray) & *(black) \\
*(gray)&*(gray)& *(gray)& *(gray) \\
*(gray)&*(gray)& *(gray) &  \\
*(gray)&*(gray)& &
\end{ytableau}};
\node (a4) at (13.5, -7.5) {$\sqcup$};
\node (a5) at (15,-7.5) {
\begin{ytableau}
*(gray)&*(gray)& *(gray)& *(gray)& *(gray) & *(gray) \\
*(gray)&*(gray)& *(gray)& *(gray)& *(gray)  \\
*(gray)&*(gray)& *(gray)& *(gray)&  \\
*(gray)&*(gray)& *(gray)&  & \\
*(gray)&*(gray)& & &
\end{ytableau}    };
\end{tikzpicture}
$$

\begin{lem}\label{L:I2alpha_beta}
Let $k\geq r\geq 1$ and $m\geq 0$.  The concatenation map $$(\lambda,\lambda')\mapsto (k+m,\lambda_1,\ldots,\lambda_{k-r}, \lambda'_1,\ldots,\lambda'_{r-1})$$ gives a bijection $$I_{k-r,r+m}\times I_{r-1,m}\rightarrow [\alpha_{k,m}(r),\beta_{k,m}(r)].$$
\end{lem}

\begin{proof}
This can been seen immediately from the stacking of skew Young diagrams.
$$
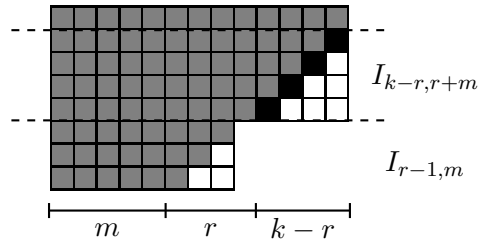
\begin{tikzpicture}
\node (a5) at (0,0) {
\begin{ytableau}
*(gray)&*(gray)&*(gray)&*(gray)& *(gray)& *(gray)& *(gray) &*(gray)&*(gray)& *(gray)& *(gray)& *(gray) & *(gray) \\
*(gray)&*(gray)&*(gray)&*(gray)& *(gray)& *(gray)& *(gray) &*(gray)&*(gray)& *(gray)& *(gray)& *(gray) & *(black) \\
*(gray)&*(gray)&*(gray)&*(gray)& *(gray)& *(gray)& *(gray) &*(gray)&*(gray)& *(gray)& *(gray)& *(black)&  \\
*(gray)&*(gray)&*(gray)&*(gray)& *(gray)& *(gray)& *(gray) &*(gray)&*(gray)& *(gray)& *(black)&  & \\
*(gray)&*(gray)&*(gray)&*(gray)& *(gray)& *(gray)& *(gray) &*(gray)&*(gray)& *(black)& &  & \\
*(gray)&*(gray)&*(gray)& *(gray)& *(gray)& *(gray) &*(gray) &*(gray) \\
*(gray)&*(gray)&*(gray)& *(gray)& *(gray)& *(gray) & *(gray) &\\
*(gray)&*(gray)&*(gray)& *(gray)& *(gray)& *(gray) & &\\
\end{ytableau}    };
\draw [dashed, thick] (-2.5,0.9) -- (2.5,0.9);
\draw [dashed, thick] (-2.5,-0.3) -- (2.5,-0.3);
\node (a2) at (3,0.25) {$I_{k-r,r+m}$};
\node (a3) at (3,-0.9) {$I_{r-1,m}$};
\draw [thick] (-2,-1.5) -- (2,-1.5);
\draw [thick] (-2,-1.6) -- (-2,-1.4);\draw [thick] (2,-1.6) -- (2,-1.4);
\draw [thick] (-0.45,-1.6) -- (-0.45,-1.4);\draw [thick] (0.75,-1.6) -- (0.75,-1.4);
\node (a2) at (-1.225,-1.75) {$m$}; \node (a2) at (0.15,-1.75) {$r$};\node (a2) at (1.375,-1.75) {$k-r$};
\end{tikzpicture}
$$
\end{proof}

\begin{prop}\label{P:tildeBkm}
$\widetilde{B}(k,m)=B(k,m)$ for all $k\geq 0$ and $m\geq 0.$
\end{prop}

\begin{proof}
First note that $\widetilde B(0,m)=B(0,m)=1$ by definition.  When $k=1$, then $I_{1,m}=\{(m+1)\}$ and $\ds \widetilde{B}(1,m)=B(1,m)=\frac{1}{(m+1)^2}$.

For $k>1$, we show that $\widetilde B(k,m)$ satisfies the same recursion as $B(k,m)$ given in Equation \eqref{Eq:B_km_recursion}.  For notational simplicity, for $\lambda\in\Lambda(k)$, let
$$\pi(\lambda):=\prod_{i=1}^k\, \lambda_i^{-1}.$$
Observe that we can factor
$$(m+k)_{k}=(m+r)\cdot(m+k)_{k-r}\cdot(m+r-1)_{r-1}$$
and hence Lemma \ref{L:I2alpha_beta} implies
\begin{align*}
\frac{\widetilde B(k-r,r+m)\widetilde B(r-1,m)}{(m+k)(m+r)}&= \frac{1}{(m+k)_{k}}\cdot\sum_{\lambda\in I_{k-r,r+m}}\left(\sum_{\lambda'\in I_{r-1,m}} \frac{\pi(\lambda)\cdot \pi(\lambda')}{(m+k)}\right)\\
&= \frac{1}{(m+k)_{k}}\cdot\sum_{\lambda\in [\alpha_{k,m}(r),\beta_{k,m}(r)]} \pi(\lambda).
\end{align*}
Applying Equation \eqref{Eq:Ikm_decomp} gives
\begin{align*}
\sum_{r=1}^k\frac{\widetilde B(k-r,r+m)\widetilde B(r-1,m)}{(m+k)(m+r)}&= \frac{1}{(m+k)_{k}}\cdot \sum_{r=1}^k\left(\sum_{\lambda\in [\alpha_{k,m}(r),\beta_{k,m}(r)]} \pi(\lambda)\right)\\
&= \frac{1}{(m+k)_{k}}\cdot \sum_{\lambda\in I_{k,m}} \pi(\lambda)\\
&=\widetilde B(k,m).
\end{align*}
Proposition \ref{P:Bkm_recursion} completes the proof.
\end{proof}

\begin{proof}[Proof of Theorem \ref{T:main2}]
Proposition \ref{P:tildeBkm} implies that $$B(k,0)=\frac{1}{k!}\cdot\sum_{\lambda\in I_{k,0}}\left(\prod_{i=1}^k\, \lambda_i^{-1}\right).$$
The theorem follows from Equation \eqref{Eq:A_km_calculation}.
\end{proof}

\bibliographystyle{amsalpha}
\bibliography{partitions}

\end{document}